\documentclass[a4paper,10pt]{amsart}
\usepackage[plainpages=false]{hyperref}
\usepackage{amsfonts,amssymb,amsthm, mathrsfs, lscape}
\usepackage{verbatim}

\usepackage{latexsym}
\usepackage{amssymb}
\usepackage{graphics}
\usepackage{graphicx}
\usepackage[space]{cite}

\usepackage{latexsym}
\usepackage{amsfonts}
\usepackage{amsmath}
\usepackage{latexsym}
\usepackage{amsfonts}
\usepackage{amssymb}
\usepackage{rawfonts}

 \usepackage[usenames,dvipsnames]{pstricks}
 \usepackage{epsfig}
 \usepackage{pst-grad} % For gradients
 \usepackage{pst-plot} % For axes
 \usepackage[space]{grffile} % For spaces in paths
 \usepackage{etoolbox} % For spaces in paths
 \makeatletter % For spaces in paths

\renewcommand{\Re}{{\operatorname{Re}\,}}
\renewcommand{\Im}{{\operatorname{Im}\,}}

\renewcommand{\epsilon}{\varepsilon}

\newcommand{\kahler}{K\"ahler }
\newcommand{\kae}{K\"ahler-Einstein }
\newcommand{\moa}{Monge-Amp\`{e}re  }

\newcommand{\PP}{{\mathbb P}}
\newcommand{\N}{{\mathbb N}}

\newcommand{\R}{{\mathbb R}}
\newcommand{\C}{{\mathbb C}}

\newcommand{\Z}{{\mathbb Z}}

\newcommand{\CP}{\C\PP}
\renewcommand{\d}{\partial}
\newcommand{\dbar}{\bar\partial}
\newcommand{\ddbar}{\partial\bar\partial}

\newcommand{\bcal}{\mathcal{B}}
\newcommand{\ccal}{\mathcal{C}}

\newcommand{\gcal}{\mathcal{G}}
\newcommand{\hcal}{\mathcal{H}}

\newcommand{\lcal}{\mathcal{L}}
\newcommand{\mcal}{\mathcal{M}}
\newcommand{\ncal}{\mathcal{N}}
\newcommand{\ocal}{\mathcal{O}}

\newcommand{\qcal}{\mathcal{Q}}

\newcommand{\vcal}{\mathcal{V}}
\newcommand{\xcal}{\mathcal{X}}

\def \xmd {X\backslash D}
\newtheorem{thm}{Theorem}[section]
\newtheorem{theorem}{{Theorem}}[section]

\newtheorem{corollary}[theorem]{{Corollary}}
\newtheorem{lem}[theorem]{{Lemma}}

\newtheorem{proposition}[theorem]{{Proposition}}

\newtheorem{remark}[theorem]{Remark}

\theoremstyle{definition}
\newtheorem{definition}[thm]{Definition}

\numberwithin{equation}{section}

\def \N {\mathbb N}

\def \C {\mathbb C}
\def \Z {\mathbb Z}
\def \R {\mathbb R}

\def \Ric {\text{Ric}}

\title[projective embedding of stably degenerating family]{projective embedding of degenerating family of K\"ahler-Einstein manifolds of negative curvature}

\author{Jingzhou Sun}
\thanks{}
\address{Department of Mathematics, Shantou University, Shantou City, Guangdong Province 515063, China}

\email{jzsun@stu.edu.cn}

\begin{document}

\begin{abstract}
    We study the Bergman embeddings of degenerating families of K\"{a}hler-Einsten manifolds of negative curvature. In one special case, we show that we can construct orthonormal bases so that the induced Bergman embeddings converge to the Bergman embedding of the limit space together with bubbles.
\end{abstract}

\maketitle

\tableofcontents

\section{Introduction}
Let $(X,\omega)$ be a \kahler manifold and let $L\to X$ be a positive line bundle endowed with a Hermitian metric $h$. The space $\hcal_k$ of $L_2$-integrable holomorphic sections of $kL$ is then a Hilbert space with the inner product defined by $$\int_X \langle s_1,s_2 \rangle_h  \omega^n.$$
$\hcal_k$  is called the Bergman space of $kL$. And when $k$ is large enough, an orthonormal basis $\{s_i\}_{i=1}^N$, $N\in \N\cup \{\infty\}$, of $\hcal_k$ induces a Kodaira embedding $\Phi_k:X\to \CP^{N-1}$, called a Bergman embedding. The Bergman embeddings have been playing a critical role as a bridge connecting \kahler geometry to algebraic geometry. For instance, using this bridge, Donaldson in \cite{donaldson2001} proved that, in the compact case, existence of a constant scalar curvature \kahler(CSCK) metric implies asymptotic Chow-stability. The Bergman kernel function, also called the density of states function, is defined as 
$$\rho_k(x)=\sum_{i=1}^{N}\parallel s_i(x)\parallel_h^2, \quad x\in X.$$ We will refer to $\rho_k$ as Bergman kernel for short.
In the compact case $\rho_k$ is very handy in the task of understanding the Bergman embeddings, largely because of the asymptotic formula as $k\to\infty$ proved by Tian \cite{Tian1990On}, Zelditch \cite{Zelditch2000Szego}, Catlin \cite{Catlin}, Lu \cite{Lu2000On}, etc.. 

Another example of the applications of Bergman embeddings and Bergman kernels is the important work of Donaldson-Sun in \cite{Donaldson2014Gromov}, in which the existence of uniform lower bounds for the Bergman kernels of families of polarized projective manifolds was shown. That was then used to prove that the Gromov-Hausdorff limits are normal algebraic varieties. One condition in \cite{Donaldson2014Gromov} is the "non-collapsing" condition, which requires that local volumes should be comparable to that of the Euclidean case. The situations that do not satisfy this condition are the "collapsing" cases. Very rare is known for the convergence of the Bergman embeddings in the collapsing case.

In \cite{sun2024projective}, the author studied the case of stably degenerating sequences of hyperbolic Riemann surfaces, which is a special  collapsing case. It was proven that on each element of the sequence, we can find suitble orthonormal bases of the Bergman spaces defined by the hyperbolic metric so that the induced sequence of Bergman embeddings basically converges to the Bergman embedding of the limit singular space with comlete hyperbolic metric on the regular part. During the convergence process, pairs of bubbles emerge and become pairs of linear $\CP^1$'s attached to the limit variety, which is the reason for using the term 'basically'. 

\

In this article, we explore the high dimensional case. To set the stage, we first recall the definition of the Cheng-Yau metric. 
Let $X$ be a smooth projective manifold of dimension $n$. 
Given a simple normal crossing divisor $D\subset X$ such that $K_X+[D]$ is ample, then it has been shown by Cheng-Yau, Kobayashi, Tian-Yau and Bando(\cite{ChengYau2, Kobayashi,TianYau3,Bando}) that the quasi-projective manifold $\xmd$ admits a unique complete \kae metric $\omega_{\text{KE}}$, known as the Cheng-Yau metric, with finite volume and $\Ric(\omega_{\text{KE}})=-\omega_{\text{KE}}$. Thus $\omega_{\text{KE}}$ defines a Hermitian metric on $K_X$ restricted to $\xmd$.

We consider the question: If we have a sequence of compact \kae manifolds that converges to some variety with Cheng-Yau metric on the regular part. Can we also construct Bergman embeddings that converges to the Bergman embedding of the limit variety? Our current exploration in this direction is in the following context. 
Recall that a degeneration of \kae manifolds is a holomorphic familty $\pi:\xcal\to B$, where $B$ is the unit disk, with the following property: 
\begin{itemize}
    \item The fibers $X_t=\pi^{-1}(t)$ are smooth except for $t=0$
    \item Each $X_t$ for $t\neq 0$ admits a \kae metric.
\end{itemize}
In the case of negative \kae metrics, Tian, in \cite{Tian1993}, proved the following theorem. 
\begin{theorem}[\cite{Tian1993}]\label{thm-tian}
    Let $\pi:\xcal\to B   $ be a degeneration of \kae manifolds $\{X_t,g_{E,t} \}$ with $\Ric(g_{E,t})=-g_{E,t}$. Assume that the total
    space $\xcal$ is smooth and the central fiber $X_0$ is the union of smooth normal crossing hypersurfaces $\{X_{0,i}\}_{1\leq i\leq m}$ in $\xcal$ with ample dualizing 
    line bundle $K_{X_0}$. Assume that no three of the divisors $X_{0,i}, 1\leq i\leq m$ have nonempty intersection. Then the \kae metrics $g_{E,t}$ on $X_t$ converges to a complete Cheng-Yau \kae metric $g_{E,0}$ on $X_0\backslash \text{Sing} (X_0)$ in the sense
    of Cheeger-Gromov: there are an exhaustion of compact subsets $F_\beta \subset\subset X_0\backslash \text{Sing}(X_0)$ and diffeomorphisms $\phi_{\beta,t}$ from $F_\beta$ into $X_t$ satifying:
    \begin{itemize}
        \item[(1)]$X_t\backslash \cup_{\beta=1}^\infty \phi_{\beta,t}(F_\beta)$ consists of a finite union of submanifolds of real codimension 1;
        \item[(2)]for each fixed $\beta$, $\phi^*_{\beta,t}g_{E,t}$ converge to $g_{E,0}$ on $F_\beta$ in $C^2$-topology on the space of Riemannian metrics as $t$ goes to $0$.
    \end{itemize}
\end{theorem}

Ruan generalized Tian's result in \cite{RuanWD1} by removing the condition that no three of the divisors $X_{0,i}, 1\leq i\leq m$ have nonempty intersection.
Later, in \cite{RuanWD2}, Ruan generalized this theorem to the toroidal case, but we will not discuss that here as we will focus on the setting in theorem \ref{thm-tian}.

In the setting of Tian's theorem, the singularity of the central fiber consists of smooth divisors of the form $D_{i,j}=X_{0,i}\cap X_{0,j}$. For simplicity, we will also denote the divisors by $D_\alpha$, $\alpha\in \Lambda$ when it is not necessary to care about $i$ and $j$.
Then $\text{Sing} (X_0)=\cup_{\alpha\in \Lambda}D_\alpha$. Denote by $|\Lambda|$ the cardinality of $\Lambda$. For each $i$, we denote by $D^i$ the union of that $D_\alpha$'s that are contained in $X_{0,i}$. For simplicity, we also denote by $D=\text{Sing} (X_0)$.

The determinant $\omega_t^n$ of the \kahler form $\omega_t$ corresponding to the \kae metric defines a Hermitian metric $h_t$ on $K_{X_t}$. By abuse of notation, we still use $h_t$ to denote the induced metric on $kK_{X_t}$.
Let $\hcal_{t,k}$ denote the Bergman space consists of $L^2$ sections $s\in H^0(X_t,kK_{X_t})$. Then $\hcal_{t,k}$ is endowed with the Hermitian inner product defined by 
$$\langle s_1,s_2 \rangle=\int_{X_t}(s_1,s_2)_{h_t}\omega_t^n, $$
for $s_1,s_2\in H^0(X_t,kK_{X_t})$. For simplicity, we will denote $K_{X_t}$ by $K_t$. 

Since $X_t$ for $t\neq 0$ is of general type, Siu's invariance of plurigenera theorem from \cite{siu1} implies that the dimension of $H^0(X_t,kK_t)$ is independent of $t$ for $t\neq 0$. We let $N_k$ denote the dimension of $H^0(X_t,kK_t)$. So for $t\neq 0$, $N_k=\dim \hcal_{t,k}$. For $t\neq 0$, $K_\xcal|_{X_t}$ is naturally isomorphic to $K_t$. The isomorphism can be explicitly written. Locally, around a point $p\in X_t$, if $(z_1,\cdots,z_n)$ are local coordinates for $X_t$, then $(t,z_1,\cdots,z_n)$ is local coordinates for $\xcal$. Let $fdz_1\wedge\cdots\wedge dz_n$ be a local section of $K_t$, then $fdt\wedge dz_1\wedge\cdots\wedge dz_n$ is a local section of $K_\xcal|_{X_t}$. Conversely, let $s$ be a local section of $K_\xcal|_{X_t}$, the inverse can be written as $\frac{s}{dt}$. The same holds on the regular part of $X_0$. 

For simplicity, we denote $K_\xcal$ by $L$. Then the push forward sheaf $\pi_* \ocal(kL)$ is a coherent sheaf on $B$, which is locally free on $B^*$. And for $k$ large enough, $\pi_* \ocal(kL)$ is locally free on $B$. For the convenience of the readers, we will include a proof of this claim at the end of the article.

Denote by $\omega_0$ the \kahler form on $X_0'=X_0\backslash \text{Sing}(X_0)$ corresponding to the complete Cheng-Yau metric. Then $\omega_0^n$ defines a Hermitian metric $h_0$ on $L|_{X_0'}$.  
Then let $\hcal_{0,k}$ denote the Bergman space of $(h_0,\omega_0^n)$-$L_2$-integrable holomorphic sections of $L^k|_{X'_{0}}$ and let $\hcal_{0,i,k}$ denote the Bergman space of $L_2$-integrable holomorphic sections of $L^k|_{X_{0,i}\backslash D^i}$.
It will be clear later that each $s\in \hcal_{0,k}$ must vanish along all the $D_\alpha$, $$\hcal_{0,k}=\oplus \hcal_{0,i,k}.$$ 
Moreover, we can identity $\hcal_{0,i,k}$ as the subspace of $\hcal_{0,k}$ consisting of sections that vanish on all $X_{0,j}$ except when $j=i$. Then clearly when $i\neq j$, $\hcal_{0,i,k}$ and $\hcal_{0,j,k}$ are mutually orthogonal.

Let $n_k=\dim \hcal_{0,k}$, $n_{i,k}=\dim \hcal_{0,i,k}$ and $d_{\alpha,k}=\dim H^0(D_\alpha,kL)$. Then we have $$N_k=n_k+\sum_{\alpha\in \Lambda}d_{\alpha,k}.$$
So we can write $\C^{N_k}=\C^{n_k}\oplus \oplus_{\alpha\in \Lambda} \C^{d_{\alpha,k}}$, which induces inclusions of the projective spaces
$$I_0:\CP^{n_k-1}\to \CP^{N_k-1},$$
and $$I_\alpha:\CP^{d_{\alpha,k}-1}\to \CP^{N_k-1}, \quad \alpha\in \Lambda.$$
For $k$ large enough, a basis of $\hcal_{0,k}$ induces an embedding $$\Phi_{0,k}:\tilde{X}_0\to \CP^{n_k-1},$$
where $\tilde{X}_0=\bigsqcup_{1\leq i\leq l}X_{0,i} $ is the desingularization of $X_0$. We denote by $\Phi_{0,i,k}$ the restriction of $\Phi_{0,k}$ to $X_{0,i}$.
So the images of $X_{0,i}$ and $ X_{0,j}$ under the embedding $\Phi_{0,k}$ are disjoint. Therefore the two copies of $D_{i,j}=X_{0,i}\cap X_{0,j}$ in $\tilde{X}_0$ have two embeddings of $D_{i,j}\to \CP^{n_k-1}$ according to the two inclusions. When we write ${D_\alpha}=D_{i,j}$, we denote by $\Phi_{k,\alpha,1}=\Phi_{0,i,k}|_{D_\alpha}$ and $\Phi_{k,\alpha,2}=\Phi_{0,j,k}|_{D_\alpha}$ the two embeddings respectively.
Our main result is the following theorem.
\begin{theorem}\label{thm-main}
    Under the setting of theorem \ref{thm-tian}, when $k$ is large enough, we can find orthonormal basis $\{s_{i,l}^0\}_{1\leq l\leq n_{i,k}}$ for $\hcal_{0,i,k}$ for each $i$, and for any sequence of points in $B^*$ that converges to $0$ we can choose a subsequence $\{t_u\}_{u=0}^{\infty}$ so that we can choose orthonormal basis $\{s_l^{t_u}\}_{1\leq l\leq N_k}$ for $\hcal_{t_u,k}$ satisfying that the images of the Bergman embeddings $\Phi_{t_u,k}:X_{t_u}\to \CP^{N_k-1}$ induced by $\{s_l^{t_u}\}_{1\leq l\leq N_k}$, as $u\to \infty$, converges to a subvariety $Y$ described as following:
    \begin{itemize}
        \item $Y$ has $2|\Lambda|+1$ irreducible components, written as $Y=Y_0\bigcup\cup_{\alpha\in \Lambda}(Y_{\alpha,1}\cup Y_{\alpha,2}) $, such that $Y_0$ is the image of $I_0\circ \Phi_{0,k}$.
        \item For each ${D_\alpha}$, there exists a basis of $H^0({D_\alpha},kL)$ that induces a Kodaira embedding $\Psi_{\alpha,k}:{D_\alpha}\to \CP^{d_{\alpha,k}-1}\subset \CP^{N_k-1}$ such that $Y_{\alpha,1}$ ($Y_{\alpha,2}$ respectively) consists of linear $\CP^1$'s connecting $\Phi_{k,a,1}(p)$($\Phi_{k,a,2}(p)$ respectively) to $ \Psi_{\alpha,k}(p) $ for all $p\in {D_\alpha}$.
    \end{itemize}
\end{theorem}

\begin{remark}
    \begin{itemize}
        \item  We will show that $Y_{\alpha,1}$ and $Y_{\alpha,2}$ are isomorphic to the projective completions of the two normal bundles of ${D_\alpha}$ in $X_{0,i}$ and $X_{0,j}$ respectively. And the production of these bubbles during the process of taking limit will be clear in the process of the proof.
        \item We suspect the necessity of having to take subsequence in the statement of the theorem. But unfortunately we have not been able to make it unnecessary.
    \end{itemize}
   
\end{remark}

During the proof of our main theorem, we need to construct orthonormal bases. So we are close to considering the Bergman kernel on $X_t$. 
Recall that in \cite{ZHOU2024109514}, Zhou showed that under that "non-collapsing" condition and lower bound on Ricci curvature of a sequence of pointed complete polarized \kahler manifolds converges in the Gromov-Hausdorff sense, then a subsequence of the polarization also converges and the Bergman kernels converge to the Bergman kernel of the polarization on the limit space. In our setting, the convergence of the Bergman kernels in the non-collapsed part is clear. What interesting is on the collapsing part. We have the following theorem.
\begin{theorem}\label{thm-2}
    Let $\rho_{t,k}$ denote the Bergman kernel of $\hcal_{t,k}$. And let $$\lambda_u(t,k)=\max_{x\in X_t}\rho_{t,k}(x),$$ and
    $$\lambda_l(t,k)=\min_{x\in X_t}\rho_{t,k}(x).$$
    Then for $k$ large enough, there are constants $c_k>0, c_k'>0$ and $c_k''>0$ such that $$c_k<\frac{\lambda_u(t,k)}{\big|\log |t|  \big|}<c_k',$$
    and  $$\lambda_l(t,k)<c_k''\big|\log |t|  \big|^{-2k+1}a_t^{2k},$$
    where $a_t=\big|\log |t|  \big|$.
\end{theorem}
The maximum is attained among the points "closest" to $D$. In fact, during the proof of theorem \ref{thm-main}, we will show that fix any smooth 
\kahler metric $g$ on $\xcal$, $\exists C_1>0, C_2>0$ such that when $|t|$ is small enough, we have $$\rho_{t,k}(p)>C_1\big|\log |t|  \big|,$$ if $d(p,D)<C_2\sqrt{|t|}$, where $d$ is the distance function defined by $g$.

\

We talk a little about the proofs. In order to prove the main theorem, we need more than theorem \ref{thm-tian} about the convergence of the \kae metrics, since it is only about the "non-collapsing part". We will also need some results in \cite{2015Collapsing} by Zhang about the "collapsing part". One may worry about the "neck" area between them. It turns out that the "neck" area does not cause much trouble in our current setting. Also, one technical strategy in the proof of our main theorem is to construct almost-orthonormal bases, whose induced Kodaira embedding is easier to describe.
Theorem \ref{thm-2} is then a byproduct of the proof of the main theorem.

\

The structure of this article is as follows. In section \ref{sec-2}, we will first recall the general estimates of the Cheng-Yau metric and of the \kae metrics on general fibers. Then, with the knowledge of these general estimates, we perform some preparatory calculations. Then in section \ref{sec-3}, we construct global sections on general fibers and show that these sections are almost orthonormal. The constructions are divided into two parts: the inner sections and the outer sections whose meaning will be clear after the constructions. Then in section \ref{sec-4}, we show the convergence of the Kodaira embeddings induced by these bases and then prove theorem \ref{thm-main} and theorem \ref{thm-2}.

\textbf{Acknowledgements.} The author would like to thank Professor Song Sun for many helpful discussions.  The author would also like to thank Professor Siarhei Finski for his interests in this work.
\section{Setting-up}\label{sec-2}
\subsection{Cheng-Yau metric.}
    In this subsection, we follow the noations in \cite{sun2024ChengYau}. The notation $D$ has different meaning only in this subsection.

    Let $X$ be a smooth projective manifold of dimension $n$. Given a simple normal crossing divisor $D\subset X$ such that $L_1=K_X+[D]$ is ample, then the quasi-projective manifold $\xmd$ admits a unique complete \kae metric $\omega_{KE}$ with finite volume and $\Ric(\omega_{KE})=-\omega_{KE}$. It is known that $\omega_{KE}$ is of Poincar\'{e} type, namely locally it is quasi-isometric to $$\frac{\sqrt{-1}dz_1\wedge d\bar{z}_1}{|z_1|^2(\log |z_1|^2)^2}+\sum_{i=2}^{n}\sqrt{-1}dz_i\wedge d\bar{z}_i,$$
    where $(z_1,\cdots,z_n)$ are local coordinates around $D$ with $z_1$ a local defining function of $D$. Also, $\omega_{KE}$ extends to be a \kahler current on $X$ and $\omega_{KE}^n$ then defines a singular Hermitian metric on $L_1$.

    Let $\hcal_k$ denote the Bergman space of $L_2$-integrable holomorphic sections of $kL_1$ on $\xmd$. Let $\hcal_{k,2}$ denote the subspace consisting of sections whose vanishing order along $D$ is larger than $1$ and let $\gcal_{k,1}$ denote its orthogonal complement.
    Then we let $\rho_k$ denote the Bergman kernel, $\rho_{k,2}$ the Bergman kernel of $\hcal_{k,2}$ and $\varrho_{k,1}$ the Bergman kernel of $\gcal_{k,1}$. 

    In \cite{sun2024ChengYau}, the author showed the asymptotic formula, around $D$, for $\rho_k$ with the Cheng-Yau metric. In this article, we will not state the main results of that article as we will not explicitly use them. We state one result therein whose idea of proof is used several times in this current article.
	Fix a smooth \kahler metric on $X$, let $d_D(p)$ be the distance from $p$ to $D$ under this smooth metric. Denote by $\tau(p)=-\log d^2_D(p) $. And let $\epsilon(k)$ denote a term that is asymptotically smaller than $k^{-N}$ for any $N$. Then we have the following theorem in the case when $D$ is smooth.
	
    \begin{theorem}[part of theorem 1.3 in \cite{sun2024ChengYau}]
        We have 
        $$\rho_{k,2}=\epsilon(k)\rho_{k},$$
        for the points where $\tau>2k$.
    \end{theorem}
    This theorem implies that when it comes very close to $D$, $\rho_{k+1}$ is dominated by $\varrho_{k,1}$. This idea is used in our construction of orthonormal bases for $\hcal_{t,k}$.

\subsection{General estimates on generic fibers}
We follow the constructions in \cite{2015Collapsing}.

The assumption in the main theorem implies that $K_\xcal$ is ample on $\xcal$. So we can choose and fix a smooth Hermitian metric $h$ on $K_\xcal$ with positive curvature. And we denote by $\omega=\frac{\sqrt{-1}}{2\pi}\Theta_h$ the \kahler form. We also have a smooth volume form $V$ on $\xcal$. Then for $t\neq 0$, we have volume form $$V_t=V\otimes (\sqrt{-1}dt\wedge d\bar{t})^{-1}.$$

%Let $W_{a}$ be the space of minimal extensions of $H^0({D_\alpha},kL)$ to $H^0(\xcal,kL)$. Let $W_{t,a}\subset \hcal_{t,k}$ be the restriction of sections in $W_a$ to $X_t$.

For each $i$, we fix a section $S_i\in H^0(\xcal,[X_{0,i}])$ such that $S_{i}$ vanishes on $X_{0,i}$. We also requires that the product $\prod_{i=1}^{m}S_i=t$. 
 Let $\parallel \cdot \parallel_i$ be a smooth Hermitian metric on $[X_{0,i}]\to \xcal$. By multiplying certain constants if necessary, we can assume that $$\parallel S_i \parallel_i\leq \epsilon \ll 1. $$

Let
$$\alpha_i=\log \parallel S_i\parallel_i^2,\quad \quad \chi_t=(\log |t|^2)^2\prod_{i=1}^{m}\alpha_i^{-2}, $$
and $$\tilde{\omega}_t=\sqrt{-1}\ddbar \log \chi_t V_t.$$
Then Zhang proved the following proposition.
\begin{proposition}[propostion 3.1 in \cite{2015Collapsing}]\label{prop-ke-t}
    Let $\varphi_t$ be the unique solution of \moa equation:
    \begin{equation}\label{eqn-ma-t}
        (\tilde{\omega}_t+\sqrt{-1}\ddbar \varphi_t)^n=e^{\varphi_t}\chi_t V_t,
    \end{equation}
    and $\omega_t=\tilde{\omega}_t+\sqrt{-1}\ddbar \varphi_t$ be the \kae metric on $X_t$. Then $$|\varphi_t|\leq C_1, \quad \text{and  } C^{-1}_2\tilde{\omega}_t\leq \omega_t\leq C_2 \tilde{\omega}_t, $$
    for constants $C_1>0$ and $C_2>0$ independent of $t$.
\end{proposition}

For a fixed component $D_\alpha$ of $\text{Sing}(X_0)$, $D_\alpha=X_{0,a}\cap X_{0,b}$ for some $a, b$. For simplicity, we can assume that $a=1, b=2$.
Around any point $p\in D_\alpha$, we can choose a neighborhood $U$ with local coordinates $(z_0,z_1,\cdots,z_n)$ so that $t=z_0z_1$ and $X_{0,1}\cap U=\{z_0=0\}$ and $X_{0,2}\cap U=\{z_1=0\}$. So $(z_2,\cdots,z_n)$ are local coordinates for $D_\alpha$ and $(z_1,\cdots,z_n)$ are local coordinates for $X_{0,1}$. 
% Since $\parallel S_1\parallel^2_1=|z_0|^2f(z)$, we can require $f(0)=1$.
Of course the choice of such coordinates is not unique. We can require that $\parallel dz_0(p)\parallel=\parallel dz_1(p)\parallel$ under the norm defined by $\omega$. We will later refer to such a neighborhood  together such coordinates an appropriate coordinates patch.
 So, since $D_\alpha$ is compact, we can cover $D_\alpha$ with finite such coordinates patchs $\{U_j\}$ and there exists a constant $C$ such that when $U_i\cap U_j\neq \emptyset$ with coordinates $(z_0,z_1,\cdots,z_n)$ and $(z'_0,z'_1,\cdots,z'_n)$ respectively, we have $$C^{-1}|z'_0|<|z_0|<C|z'_0|,$$ hence also $C^{-1}|z'_1|<|z_1|<C|z'_1|$. We can also modify the choice of $(z_2,\cdots,z_n)$ so that there exists a constant $C'$ such that $$\frac{1}{C'}<\frac{V}{(\sqrt{-1})^{n+1} dz_0\wedge d\bar{z}_0\wedge\cdots dz_n\wedge d\bar{z}_n }<C',$$
on every $U_j$. We denote by $\mu_\alpha=\omega^{n-1}$ the volume form on $D_\alpha$. Then there exists a constant $C''$ such that $$\frac{1}{C''}<\frac{\mu_\alpha}{(\sqrt{-1})^{n-1} dz_2\wedge d\bar{z}_2\wedge\cdots dz_n\wedge d\bar{z}_n }<C'',$$
on every $U_j\cap D_\alpha$. We can furthermore require that, under the Hermitian inner product defined by $\omega$, $\langle dz_i(p),dz_j(p) \rangle =\delta_j^i$ for $i\geq 2,j\geq 2$. On $X_{0,1}$ (or similarly on $X_{0,2}$), we can choose a defining function $w_1$ of $D_\alpha$ so that $(w_1,w_2=z_2,\cdots,w_n=z_n)$ are local coordinates and $\omega(p)=\sum \delta_j^i dw_i\wedge d\bar{w}_j$. By a compactness argument, one easily sees that there is a positive lower bound for the angle between $\frac{\partial}{\partial z_1}$ and $\frac{\partial}{\partial w_1}$ on $U_j\cap X_{0,1}$ valid for every $U_j$. The same holds for $X_{0,2}$.

For simplicity, we can also make $U$ a polydisc by requiring $|z_i|<R_i, i=0,\cdots, n$. 

Since $dt=z_1dz_0+z_0dz_1$, we have $dt\wedge dz_1=z_1dz_0\wedge dz_1$ and $dt\wedge dz_0=-z_0dz_0\wedge dz_1$. So for $t\neq 0$, 
$$\frac{(\sqrt{-1})^{n+1} dz_0\wedge d\bar{z}_0\wedge\cdots dz_n\wedge d\bar{z}_n}{\sqrt{-1}dt\wedge d\bar{t}}=\frac{(\sqrt{-1})^{n}}{|z_1|^2} dz_1\wedge d\bar{z}_1\wedge\cdots dz_n\wedge d\bar{z}_n,$$
and also$$\frac{ dz_0\wedge d\bar{z}_0\wedge\cdots dz_n\wedge d\bar{z}_n}{dt\wedge d\bar{t}}=\frac{dz_0\wedge d\bar{z}_0}{|z_0|^2} \wedge dz_2\wedge d\bar{z}_2\wedge\cdots dz_n\wedge d\bar{z}_n.$$
Therefore there exists another constant $C$ such that
$$\frac{1}{C}<\frac{|z_1|^2V_t}{(\sqrt{-1})^{n} dz_1\wedge d\bar{z}_1\wedge\cdots dz_n\wedge d\bar{z}_n }<C,$$
on each $U_i$. We denote by 
$$\tilde{\chi}_t=\frac{(\log |t|^2)^2}{\alpha_1^2\alpha_2^2}. $$
Then there is a constant $C'$ such that 
$$C^{'-1}<\frac{\chi_t}{\tilde{\chi}_t}<C',$$on each $U_i$.

We denote by $\sigma=-\log |z_1|^2$, then $\alpha_2=\psi_2-\sigma$ and $\alpha_1=\log |t|^2+\sigma+\psi_1$ for some bounded functions $\psi_1,\psi_2$. Then $\tilde{\chi}_t^{-1}=(\frac{\alpha_1}{\log |t|^2})^2(\psi_2-\sigma)^2$. So we have the following proposition.
\begin{proposition}
    There exist positive constants $t_0$, $M$ and $C_3$ such that for $0<|t|<t_0$, $M<\sigma<\frac{-1}{2}\log |t|^2$, we have 
    $$C^{-1}_3<\frac{|z_1|^2\sigma^2\chi_t V_t}{(\sqrt{-1})^{n} dz_1\wedge d\bar{z}_1\wedge\cdots dz_n\wedge d\bar{z}_n}<C_3,$$on each $U_i$.
\end{proposition}
Then together with proposition \ref{prop-ke-t}, we have the following.
\begin{corollary}\label{cor-wt-bound}
    There exist positive constants $t_0$, $M$ and $C_4$ such that for $0<|t|<t_0$, $M<\sigma<\frac{-1}{2}\log |t|^2$, we have 
    $$C^{-1}_4<\frac{|z_1|^2\sigma^2\omega_t^n}{(\sqrt{-1})^{n} dz_1\wedge d\bar{z}_1\wedge\cdots dz_n\wedge d\bar{z}_n}<C_4,$$
    where $\omega_t$ is the \kae metric on $X_t$, on each $U_i$.
\end{corollary}
We can write $$\omega_t^n=e^{-\phi_t}\frac{(\sqrt{-1})^{n}}{|z_1|^2}dz_1\wedge d\bar{z}_1\wedge\cdots dz_n\wedge d\bar{z}_n.$$
Let $s=f(dz_0\wedge\cdots\wedge dz_n)^{\otimes k}$ be a local section of $K_\xcal^{k}$. 
Then we have a section $$s_t\triangleq \frac{s}{(dt)^{\otimes k}}=\frac{f}{z_1^{ k}}(dz_1\wedge\cdots\wedge dz_n)^{\otimes k}$$
in $H^0(X_t,kK_{X_t})$.

\textbf{Terminology.} For simplicity, we will call $s_t$ the restriction of $s$ to $X_t$.

So the pointwise norm of $s_t$ defined by the \kae metric on $X_t$ is
$$\parallel s_t\parallel^2_{\text{KE}}=|f|^2e^{k\phi_t}.$$
And so its local $L_2$ norm is given by 
$$\int |f|^2e^{(k-1)\phi_t}\frac{(\sqrt{-1})^{n}}{|z_1|^2}dz_1\wedge d\bar{z}_1\wedge\cdots dz_n\wedge d\bar{z}_n.$$
Therefore, by corollary \ref{cor-wt-bound}, we need to estimate integrals of the form
$$\int |f(z_1)|^2(\log |z_1|^2)^{2k}\frac{\sqrt{-1}}{|z_1|^2}dz_1\wedge d\bar{z}_1.$$
In particular, when the region is $\Omega=\{z_1| M<\sigma<\frac{-1}{2}\log |t|^2\}$,
\begin{equation}\label{e-f-Omega}
    \int_\Omega |f(z_1)|^2(\log |z_1|^2)^{2k}\frac{\sqrt{-1}}{|z_1|^2}dz_1\wedge d\bar{z}_1=\pi\int_M^y |f(z_1)|^2(\sigma)^{2k}d\sigma,
\end{equation}
where $y=\frac{-1}{2}\log |t|^2$. When $f=1$, we have 
$$\int_M^y (\sigma)^{2k}d\sigma=\frac{y^{2k+1}-M^{2k+1}}{2k+1}.$$
So \begin{eqnarray}
    \int_{\sqrt{y}}^y (\sigma)^{2k}d\sigma &=&\frac{y^{2k+1}-y^{(2k+1)/2}}{2k+1}\\\label{e-y-sqrty}
    &=&\frac{1-y^{-(2k+1)}/2}{1-(M/y)^{2k+1}} \int_M^y (\sigma)^{2k}d\sigma.
\end{eqnarray}

Then we have 
$$\lim_{t\to 0}\frac{\int_{\sqrt{y}}^y (\sigma)^{2k}d\sigma}{\int_M^y (\sigma)^{2k}d\sigma}=1.$$
When $f=z_1^i$ for $i\geq 1$, we do not have explicit formulas for the integral

\subsection{Collaping part.}

Denote by $G_J\subset \C$ the strip $J\times \sqrt{-1}\R$ for any subset $J\subset (0,1)$. Let $w$ be the complex coordinate for $G_J$ and let $x=\Re w$. Let
$J_t=(\frac{\log \epsilon}{\log |t|},1-\frac{\log \epsilon}{\log |t|})$, we denote by $G_t=G_{J_t}$.
Define a covering map:
$$P_t:G_t\times \{(z_2,\cdots,z_n)\big| |z_i|<R_i, i\geq 2  \}\to U\cap X_t$$
by $$P_t(w,z_2,\cdots,z_n)=(e^{w\log |t|},z_2,\cdots,z_n).$$

In \cite{2015Collapsing}, Zhang proved the following results. 
\begin{lem}[lemma 3.2 in  \cite{2015Collapsing}]\label{lem-zhang-3.2}
    Let $K\subset (0,1)$ be a compact subset such that $K\subset J_t$ for $|t|\ll 1$. On $G_K\times (U\cap D_\alpha)$, when $t\to 0$,
    $$P_t^*\chi_t V_t\to V_0'=\frac{dw\wedge d\bar{w}}{16(1-x)^2x^2}\wedge V_\alpha, $$
    in the $C^{\infty}$-sense, where $V_\alpha$ is a smooth volume form on $U\cap D_\alpha$.
\end{lem}
When $U=U_j$ for some $j$, $V_{\alpha,j}$ is a more accurate notation for the volume form $V_{\alpha}$ in the lemma. By our requirements on the coordinates patches $U_j$, we have that on the intersection of $(U_i,(z_0,
\cdots,z_n))$ and $(U_j,(z'_0,
\cdots,z_n'))$, we have $$\big|\log \frac{|z_1'|}{|z_1|} \big|<C,$$ for some constant $C>0$. So the corresponding coordinates on $K$ satisfies $|x'-x|<\frac{C}{|\log |t||}$. So when $t\to 0$, $x'-x\to 0$. And the imaginary part is similar. Therefore, we have $$V_{\alpha,j}=V_{\alpha,i},$$
on overlaps, namely they glue together to produce a global volume form $\vcal_\alpha$ on $D_\alpha$. And we denote by $\vcal$ the induced volume form on $D$. And for simplicity, we denote by $\vcal_j$ when restricted to $U_j\cap D_\alpha$ for some $\alpha$.

\begin{theorem}[lemma 3.3 in  \cite{2015Collapsing}]\label{the-zhang-3.3}
Let $\varphi_t$ be the unique solution of (\ref{eqn-ma-t}), and $\omega_t=\tilde{\omega}_t+\sqrt{-1}\d\dbar \varphi_t$. For any sequence $t_k\to 0$, a subsequence of $\phi_{t_k}\circ P_{t_k}$ converges to $\phi_0$ in the $C^\infty $-sense on $K\times \sqrt{-1}\R\times (U\cap D_\alpha)$ satisfying the complex \moa equation:
\begin{equation}\label{eqn-ma-0}
    (\tilde{\omega}_0+\sqrt{-1}\ddbar \varphi_0)^n=e^{\varphi_0} V_0'
\end{equation}
with $|\varphi_0|\leq C_3$, and $C_4^{-1}\tilde{\omega}_0\leq \tilde{\omega}_0+\sqrt{-1}\ddbar \varphi_0\leq C_4\tilde{\omega}_0$.

Furthermore, $\varphi_0$ is independent of $\Im(w)$, i.e. 
$$\varphi_0=\varphi_0(x,z_{2},\cdots,z_n).$$
\end{theorem}
By lemma \ref{lem-zhang-3.2} and theorem \ref{the-zhang-3.3},
$e^{\varphi_0} V_0'=e^{\varphi_0}\frac{dw\wedge d\bar{w}}{16(1-x)^2x^2}\wedge\vcal_i$ on $U_i\cap D_\alpha$. 
\begin{definition}
    We call a sequence $\{t_l\}$ that converges to 0 a good sequence if $P_t^*\omega_t^n$ converges on $G_K\times (U_i\cap D_\alpha)$ for every $\alpha,i$ and for every $K$.
\end{definition}
Let $t_u\to 0$ be a good sequence. We consider the form $e^{\varphi_0}\frac{\vcal_i}{16(1-x)^2x^2}$ on 
$G_K\times (U_i\cap D_\alpha)$. If we change the coordinates $(z_2,\cdots,z_n)$ on $U_i\cap D_\alpha$ to $(z_2',\cdots,z_n')$, then the representation function of $\omega_t^n$ is changed by $|\det J|^2$, where $J$ is the matrix $ \{\frac{\partial z_a}{\partial z_b'} \}_{a,b\geq 2}$. Then since the image of $G_K\times (U_i\cap D_\alpha)$ under $P_t$ converges to $U_i\cap D_\alpha$ as $t\to 0$, the limit of the representation function of $P_t^*\omega_t^n$ is changed by $|\det J|_{U_i\cap D_\alpha}|^2$. Therefore, $e^{\varphi_0}\frac{\vcal_i}{16(1-x)^2x^2}$ glues together to be a smooth section of $\Pi^{-1}\Omega_{D_\alpha}^{n-1,n-1}$ on $G_K\times  D_\alpha$, where $\Pi:G_K\times  D_\alpha\to D_\alpha$ is the natural projection. And by letting $K=[\delta,1-\delta]$ and $\delta\to 0$, we get a smooth section of $\Pi^{-1}\Omega_{D_\alpha}^{n-1,n-1}$ on $G_{(0,1)}\times  D_\alpha$.

Recall that 
$$\omega_t^n=e^{-\phi_t}\frac{(\sqrt{-1})^{n}}{|z_1|^2}dz_1\wedge d\bar{z}_1\wedge\cdots dz_n\wedge d\bar{z}_n,$$ then since $P_t^*dz_1\wedge d\bar{z}_1=(\log |t|)^2|z_1|^2 dw\wedge d\bar{w}$, we have $$P_t^*\omega_t^n=(\log |t|)^2e^{-\phi_t\circ P_t}(\sqrt{-1})^{n}dw\wedge d\bar{w}\wedge dz_2\wedge d\bar{z}_2\wedge\cdots dz_n\wedge d\bar{z}_n.$$For simplicity, we denote by $d\mu=(\sqrt{-1})^{n}dw\wedge d\bar{w}\wedge dz_2\wedge d\bar{z}_2\wedge\cdots dz_n\wedge d\bar{z}_n$. Then if we denote by $\vcal_i=e^{\psi_i}(\sqrt{-1})^{n-1} dz_2\wedge d\bar{z}_2\wedge\cdots dz_n\wedge d\bar{z}_n$, then 
$$\lim_{l\to\infty }(\log |t_l|)^2e^{-\phi_{t_l}\circ P_{t_l}}=\frac{1}{16(1-x)^2x^2} e^{\psi_i+\varphi_0}.$$
Let $s=f(dz_0\wedge\cdots\wedge dz_n)^{\otimes (k+1)}$ be a local section of $K_\xcal^{k+1}$. Then the $L_2$ norm over $P_t(G_K\times (U_i\cap D_\alpha))$ of $s_t$ is $$\int_{K\times\sqrt{-1}[0,\frac{2\pi}{-\log |t|}]\times (U_i\cap D_\alpha)} |f|^2e^{k\phi_t}(\log |t|)^2d\mu.$$
In particular, when $f$ depends only on $(z_2,\cdots,z_n)$, we are looking at 
$$\int_{K\times \sqrt{-1}[0,\frac{2\pi}{-\log |t|}]\times (U_i\cap D_\alpha)}|f|^2 \frac{e^{k\phi_t}}{(\log |t|)^{2k}}(\log |t|)^{2k+2}d\mu.$$

So its quotient by $-(\log |t_l|)^{2k+1}$, as $l\to \infty$, converges to
\begin{equation}\label{e-0-D}
    \pi\int_{K\times (U_i\cap D_\alpha)}|f|^2 [4(1-x)x]^{2k}e^{-k\psi_i-k\varphi_0} d\mu=\int_{U_i\cap D_\alpha}|f|^2 e^{-(k+1)\psi_i} d\nu_{\{t_l\},\alpha,K,k}
\end{equation}
where \begin{equation}
    d\nu_{\{t_l\},\alpha,K,k}=\pi\vcal\int_{x\in K}^{}e^{-k\varphi_0}[4(1-x)x]^{2k}dx,
\end{equation}
is a volume form on $D_\alpha$. We also denote by 
$$d\nu_{\{t_u\},\alpha,k}=\pi\vcal\int_{0}^{1}e^{-k\varphi_0}[4(1-x)x]^{2k}dx,$$ 
the induced volume form on $D_\alpha$.

Note that $e^{-\psi_i}=\parallel dz_2\wedge\cdots\wedge dz_n\parallel^2_{\vcal_i}$. So
the term $|f|^2e^{-(k+1)\psi_i}$ can be understood as the point-wise norm of a pluri-canonical form. To connect $K_{\xcal}$ to $K_{D_\alpha}$, we notice that $dz_0\wedge dz_1$ on $D_\alpha$ does not depend on our choice of $(z_0,z_1)$. Indeed, since we require $z_0z_1=t$, a different choice $(z_0',z_1')$ must satisfy $z_0'=z_0F(z), z_1'=\frac{z_1}{F(z)}$. Let $q\in U\cap D_\alpha$, $c=F(q)$. Then, at $q$, we have $dz_0'=cdz_0$ and $dz_1'=c^{-1}dz_1$. So $dz_0\wedge dz_1=dz_0'\wedge dz_1'$. So we have an isomorphism $K_{\xcal}|_{D_\alpha}\to K_{D_\alpha}$ given by $dz_0\wedge\cdots\wedge dz_n\to dz_2\wedge\cdots\wedge dz_n$. In the following, we will tacitly use this isomorphism. 

For later use, we denote by \begin{eqnarray*}
    b_k&=&\int_0^1 [4(1-x)x]^{2k}dx\\
    &=&16^k \frac{\Gamma(2k+1)^2}{\Gamma(4k+1)}.
\end{eqnarray*}

\section{Construction of almost orthonormal bases}\label{sec-3}
\subsection{Inner sections}
To apply these estimates to global sections, we use an extension theorem by Finski. Let $X$ be a complex manifold with a positive line bundle $(E,h^E)$ over it. Let $(F,h^F)$ be an arbitrary Hermitian vector bundle over $X$. Let $Y$ be a complex manifold and let $\iota: Y\to X$ be an embedding. Let $n=\dim X$ $m=\dim Y$. Fix volume forms $dv_X$ and $dv_Y$ on $X$ and $Y$ respectively. Then both $H^0(X,E^k\otimes F)$ and $H^0(Y,\iota^*(E^k\otimes F) )$ is endowed with a Hermitian inner product. Denote by $H^{0,\bot }(X,E^k\otimes F)$ the orthogonal complement of the space of sections that vanish on $Y$. Then we would like to estimate the norm of the restriction operator
$$\text{Res}_k:H^{0,\bot }(X,E^k\otimes F)\to H^0(Y,\iota^*(E^k\otimes F) ).$$

Under the assumption of bounded geometry, Finski proved the following theorem.

\begin{theorem}[theorem 4.1 in \cite{Finski}]\label{thm-fin-total}
    There exist $c,C>0$ and integer $k_1>0$ such that for any $k\geq k_1$, we have
   $$ ck^{\frac{n-m}{2}}\leq \parallel \text{Res}_k\parallel\leq Ck^{\frac{n-m}{2}}.  $$
\end{theorem}
The readers are refered to \cite{Finski} (Definitions 2.3 and 2.4) for the definition of bouded geometry. 
The lower bound in the theorem has a less abstract form:
\begin{theorem}[theorem 4.4 in \cite{Finski}]
    There exists $C>0$, $k_1>0$ such that for any $k\geq k_1$ and $s\in H^0(Y,\iota^*(E^k\otimes F) )$, there is $f\in H^0(X,E^k\otimes F)$ such that $f|_Y=s$ and 
    $$\parallel f\parallel_{L^2(X)}\leq Ck^{(m-n)/2}\parallel s\parallel_{L^2(Y)}$$
\end{theorem}
We let $Y=D_\alpha$, $X=\xcal$ with \kahler form $\omega$ and $E=K_\xcal$ with the Hermitian metric $h$. 
Then by slightly shrinking the disk $B$, our setting with $F$ the trivial bundle satisfies the assumption of bounded geometry in \cite{Finski}. We denote by $\parallel \cdot \parallel_{h}$ the $L_2$-norm defined by $h$ and $\omega$ for sections in $H^0(D_\alpha,kL)$. Then let $s\in H^0(D_\alpha,kL)$ be of unit norm, we have an extension $\tilde{s}\in H^0(\xcal,kL)$
For any submanifold $Y'\subset \xcal$, We denote by $\parallel \cdot \parallel_{h,Y'}$ the $L_2$-norm defined by $h$ and $\omega$ for sections in $H^0(Y',kL)$. Then there exists $C>0, k_1>0$ such that for a section $s\in H^0(D_\alpha,kL)$ of unit norm with $k\geq k_1$, we have an extension $\tilde{s}\in H^0(\xcal,kL)$ such that $$\parallel \tilde{s}\parallel^2_{h,\xcal}\leq \frac{C_1}{k^2}.$$
To fix our choice, we let $\tilde{s}$ to be of minimal norm.

Denote by $\tilde{s}_{0,i}$ the restriction of $\tilde{s}$ to $X_{0,i}$. Then for $i=1,2$, $\tilde{s}$ is an extension of $\tilde{s}_{0,i}$ from $X_{0,i}$. Then by the upper bound in theorem \ref{thm-fin-total}, we have 
\begin{eqnarray*}
    \parallel \tilde{s}_{0,i}\parallel^2_{h,X_{0,i}}&\leq& C_2k\parallel \tilde{s}\parallel^2_{h,\xcal}\\
    &\leq& \frac{C_3}{k},
\end{eqnarray*}
for $i=1,2$.

If we fix a \kahler metric on $\xcal$, then any vector field $v$ on $B$ can be horizontally lifted to be a vector field $\tilde{v}$ on $\xcal\backslash \text{Sing}(X_0)$. Then for any $t\neq 0$, we can use a straight line $\gamma(u)$ on $B$ to connect $0$ and $t$, and then the integral flow $\tilde{v}$ with $v=\frac{\partial}{\partial u}$ gives the maps $\phi_{\beta,t}$ in Tian's theorem, at least for $t$ small. %For $|z_1|>\delta_1$, for $|t|$ small enough, we have $|z_1|>\delta_1+\delta_1^2$ on $X_t$ is contained in the image of $\phi_{\delta_1,t}$ with a slightly smaller polydisc $U'\subset U$.

Denote by $\tilde{s}_t$ the restriction of $\tilde{s}$ to $X_t$. 
For every $\delta_1>0$, we let $U_i^{\delta_1}$ denote the subset of $U_i$ consists of points satisfying $|z_0|>\delta_1$ or $|z_1|>\delta_1$. Then $\exists M_1>0$ such that $\int_{U_j^{\delta_1}\cap X_{0,i}}\parallel \tilde{s}_{0,i}\parallel^2_{CY}<\epsilon_1$. Then
by theorem \ref{thm-tian}, we have $\int_{U_j^{\delta_1}\cap X_t }\parallel \tilde{s}_{t}\parallel^2_{\text{KE}}<2M_1,$ for $|t|$ small enough. Let $M$ be the constant as in the statement of corollary \ref{cor-wt-bound}. We can choose $M>2k$. And we denote by $y_1=-\log |t|^2-M$. We choose $\delta_1$ so that $-\log |\delta_1|^2=4M$.

On $U_j$, $\tilde{s}=f(dz_0\wedge\cdots\wedge dz_n)^{\otimes (k+1)}$. Then on 
$U_j\cap X_t$, using $(z_1,\cdots,z_n)$ as coordinates, $f=\sum_{i\in \Z }a_i z_1^i$, where $a_i=a_i(z_2,\cdots,z_n)$. We denote by $f_+=\sum_{i>0 }a_i z_1^i$ and $f_-=\sum_{i<0 }a_i z_1^i$. So now we need to consider
$$\int |z_1|^{2i}(\sigma)^{2k}d\sigma=\int e^{-i\sigma+2k\log \sigma}d\sigma.$$
We denote by $\tilde{s}_+=f_+(dz_0\wedge\cdots\wedge dz_n)^{\otimes (k+1)}$ and $\tilde{s}_-=f_-(dz_0\wedge\cdots\wedge dz_n)^{\otimes (k+1)}$.
% Recall that $y=-\log |t|$.

To analyze the integrals, we recall the following basic lemma from \cite{sun2024ChengYau}.
\begin{lem}\label{lem-concave}
    Let $f(x)$ be a concave function. Suppose $f'(x_0)<0$, then we have
    $$\int_{x_0}^\infty e^{f(x)}dx\leq\frac{e^{f(x_0)}}{-f'(x_0)}$$
\end{lem}

Since the exponent $-i\sigma+2k\log \sigma$ is a concave function of $\sigma$, we have 
$$\int_{4M}^{y_1} e^{-i\sigma+2k\log \sigma}d\sigma<(i-\frac{2k}{4M})^{-1}e^{-i4M+2k\log (2M) }.$$
Also clearly we have 
$$\int_M^{4M}e^{-i\sigma+2k\log \sigma}d\sigma>e^{-i(M+1)+2k\log(M+1) }.$$
Therefore $$\int_{M}^{y_1} e^{-i\sigma+2k\log \sigma}d\sigma<2\int_M^{4M}e^{-i\sigma+2k\log \sigma}d\sigma.$$

So for general holomorphic function $g=\sum_{i=1}^{\infty}a_iz_1^i$, we have 
\begin{eqnarray*}
   \int_{M\leq \sigma \leq y_1} |g|^{2}(\sigma)^{2k}\frac{\sqrt{-1}}{|z_1|^2}dz_1\wedge d\bar{z}_1 &=&\pi\sum_{i=1}^{\infty}|a_i|^2\int_M^{y_1} |z_1|^{2i}(\sigma)^{2k}d\sigma\\
   &\leq &2\int_{M\leq \sigma \leq 4M} |g|^{2}(\sigma)^{2k}\frac{\sqrt{-1}}{|z_1|^2}dz_1\wedge d\bar{z}_1
\end{eqnarray*}

Notice that for $\sigma>y$, we have $(\log |z_0|^2)^2=(2y-\sigma)^2<\sigma^2$. So we can apply the preceeding calculations to $f_+$ and to $f_-$ with $z_0$ replacing $z_1$ to get that \begin{eqnarray}\label{e-spm}
    \int_{U_j\cap X_t }\parallel (\tilde{s}_{\pm })_{t}\parallel^2_{\text{KE}}\omega_t^n&\leq &C^k\int_{U_j^{\delta_1}\cap X_t }\parallel (\tilde{s}_{\pm })_{t}\parallel^2_{\text{KE}}\omega_t^n\\ \label{ine-st-i}
    &\leq &2C^kM_1
\end{eqnarray}
for some contant $C$ independent of $t$ and $j$.
By the calculations preceeding equation \ref{e-0-D}, and since $\varphi_0$ is uniformly bounded, we get that $\exists c>0$ such that $$\int_{U_j\cap X_t }\parallel (\tilde{s}-\tilde{s}_{+ }-\tilde{s}_{- })_t\parallel^2_{\text{KE}}\omega_t^n>c^k b_k|\log |t||^{2k+1}\int_{U_j\cap D_\alpha}|a_0|^2e^{-(k+1)\psi_j}\vcal_j,$$ for every $j$. Therefore, \begin{equation}\label{e-tils-pm}
    \int_{U_j\cap X_t }\parallel \tilde{s}_t\parallel^2_{\text{KE}}\omega_t^n>c_k'\big|\log |t|\big|^{2k+1}\int_{U_j\cap D_\alpha}|a_0|^2e^{-(k+1)\psi_j}\vcal_j,
\end{equation}
 for $t$ large enough, for a constant $c_k'$. 

\

Let $d(q)$ denote the distance of a point $q$ to $D_\alpha$ defined by $\omega$, and let $\tau(q)=-\log d^2(p)$. In an appropriate coordinates patch $U_j$, we have $d^2(q)\approx (z_0,z_1)Q(\bar{z}_0,\bar{z}_1)^t$ with $Q$ a positive definite Hermitian matrix depending on $(z_2,\cdots,z_n)$. So $\exists C$ independent of $j$ such that $$\frac{1}{C}(|z_0|^2+|z_1|^2)<d^2(q)<C(|z_0|^2+|z_1|^2).$$
So $\exists C_1>0$ such that $$-C_1<\tau-\log (|z_0|^2+|z_1|^2)<C_1.$$
So when $|z_1|\geq |z_0|$, we have $$-C'_1<\tau-\sigma<C'_1.$$

So the $L_2$-norm of $\tilde{s}_t$ over the region $\ccal$ where $\tau>\frac{3}{2}M$ satisfies $$\parallel \tilde{s}_t\parallel^2_{\text{KE},\ccal}>c_k \big|\log |t|\big|^{2k+1},$$ for some constant $c_k$ depending on $k$. We denote by $\tilde{s}_{tn}=\frac{1}{\parallel \tilde{s}_t\parallel_{\text{KE},\ccal}}\tilde{s}_t$. Then by inequality \ref{ine-st-i} and equation \ref{e-y-sqrty}, we get that the total $L_2$-norm defined by $\omega_t$ over the region $\mcal_j=\{\sigma\big|M<\sigma<3M \text{ or } M<2y-\sigma<3M \}\subset U_j$ is $<\big|\log |t|\big|^{\frac{k}{2}} $.  Let $\ncal_\alpha$ denote the region $q\in\xcal$ satisfying $\frac{3}{2}M<\tau(q)<\frac{5}{2}M$. Then we have $\ncal_\alpha\cap U_j\subset \mcal_j$ for every $j$.
So we have 
\begin{proposition}
    There exists $c>0$ such that for $|t|$ small enough, we have
    $$\int_{\ncal_\alpha}\parallel \tilde{s}_{tn}\parallel^2_{\text{\text{KE}}}\omega_t^n<c\big|\log |t|\big|^{-k} $$
\end{proposition}
To not deal with the points not close to $D_\alpha$, we use H\"ormander's $L^2$ estimates. The following lemma is well-known, see for example \cite{Tian1990On}. 

\begin{lem}
    Suppose $(M,g)$ is a complete \kahler manifold of complex dimension $n$, $\mathcal L$ is a line bundle on $M$ with Hermitian metric $h$. If 
    $$\langle-2\pi i \Theta_h+Ric(g),v\wedge \bar{v}\rangle_g\geq C|v|^2_g$$
    for any tangent vector $v$ of type $(1,0)$ at any point of $M$, where $C>0$ is a constant and $\Theta_h$ is the curvature form of $h$. Then for any smooth $\mathcal L$-valued $(0,1)$-form $\alpha$ on $M$ with $\bar{\partial}\alpha=0$ and $\int_M|\alpha|^2dV_g$ finite, there exists a smooth $\mathcal L$-valued function $\beta$ on $M$ such that $\bar{\partial}\beta=\alpha$ and $$\int_M |\beta|^2dV_g\leq \frac{1}{C}\int_M|\alpha|^2dV_g$$
    where $dV_g$ is the volume form of $g$ and the norms are induced by $h$ and $g$.
\end{lem}
In our setting with $(M,g)=(X_t,\omega_t)$ and the line bundle is $kL$, so for 
$k$ large, the assumption of the Lemma is satisfied.

We define a cut off function $\varpi(\eta) $ of one variable satsifying the following:
\begin{itemize}
    \item $\varpi(\eta)=1 $ for $\eta\geq \frac{5}{2}M$;
    \item $\varpi(\eta)=0 $ for $\eta\leq \frac{3}{2}M$;
    \item $0\leq \varpi'(\eta)<1/2 $.
\end{itemize}  
Then we want to solve the equation $$\dbar v=\dbar\varpi(\tau) \otimes\tilde{s}_{tn},$$ on $X_t$. By theorem \ref{thm-tian}, one sees that on the region where $M<\sigma<3M $ on $U_j$, $\exists C_1>0$ such that $$|\dbar\varpi(\tau)|<C_1 |\dbar \sigma|. $$
So $|\dbar\varpi(\tau)|^2<C_2M^2$ for some $C_2>0$.
And similarly for the region where $ M<2y-\sigma<3M $.
So we have $$\int_{X_t}\parallel \dbar\varpi(\tau) \otimes\tilde{s}_{tn}\parallel^2_{\text{KE}}\omega_t^n<c_M\big|\log |t|\big|^{-k},  $$
for some constant $c_M$ depending only on $M$. Therefore, we can find a solution $v\in C^{\infty}(X_t,kL)$ satisfying  
$$\int_{X_t}\parallel v\parallel^2_{\text{KE}}\omega_t^n<c_M\big|\log |t|\big|^{-k}. $$
And we get a new holomorphic pluri-canonical form $\triangleq \varpi(\tau)\tilde{s}_{tn}-v$. Then we define
$$\tilde{s}_{t,\text{mod}}\triangleq \parallel \tilde{s}_t\parallel_{\text{KE},\ccal}(\varpi(\tau)\tilde{s}_{tn}-v) .$$
If we replace $\sqrt{y}$ in equation \ref{e-y-sqrty} by $\sqrt{y}+c$, similar estimate holds. So we also have 

\begin{proposition}\label{prop-sqrty}
    $$\int_{X_t,\tau<\sqrt{y}}\parallel \tilde{s}_{t,\text{mod}}\parallel^2_{\text{KE}}\omega_t^n=O(\big|\log |t|\big|^{k+1})  $$
    
\end{proposition}
\begin{remark}
    We should mention that till now, for simplicity, we have been using $\tau(d(q))$  for the distance to $D_\alpha$ for a fixed $\alpha$. Later we will also use $\tau(d(q))$ for the distance to $D$. At that time, we will refer to the former one as $\tau_\alpha$.  
\end{remark}

Now we calculate $\parallel \tilde{s}_{t_l}\parallel_{\text{KE},\ccal}$ when $\{t_l\}$ is a good sequence. And for simplicity, we shrink $U_j$ so that $-\log |z_0|^2>M$ and $-\log |z_1|^2>M$ for $z\in U_j$. 

Since the image $P_t(G_K\times (U_i\cap D_\alpha))$, where $K=[\delta,1-\delta]$, is $\{z\in U_j\cap X_t\big| -\delta\log |t|<-\log |z_1|<-(1-\delta)\log |t|   \}$, and since $$\lim_{\delta\to 0}\frac{\int_{M}^{-\delta\log |t|  }\sigma^{2k}d\sigma}{\int_{M}^{-\frac{1}{2}\log |t|  }\sigma^{2k}d\sigma}=0,$$ by symmetry, we have that $$\lim_{l\to \infty}\frac{-1}{(\log |t_l|)^{2k+1}}\int_{U_j}\parallel \tilde{s}_{t_l}\parallel_{\text{KE}}^2 \omega_t^n=\int_{U_j\cap D_\alpha}\parallel s(p) \parallel^2_{\vcal}d\nu_{\{t_l\},\alpha,k}.  $$
The same holds for any sub-polydisc of $U_j$. Therefore, we get the following conclusion.
\begin{proposition}\label{prop-inner}
    Let $\{t_l\}$ be a good sequence. Let $N_{\alpha,t}=\{q\in X_t\big|\tau(q)>\frac{3}{2}M  \}$ and let $$\parallel s \parallel_{\{t_l \},D_\alpha}$$ denote the $L_2$-norm defined by $\vcal$ and $d\nu_{\{t_l\},\alpha,k}$, namely,
    $$\parallel s \parallel_{\{t_l \},D_\alpha}^2=\int_{D_\alpha}\parallel s(p) \parallel^2_{\vcal}d\nu_{\{t_l\},\alpha,k}$$
    for $s\in H^0(D_\alpha,(k+1)L)$. Then we have 
    $$\lim_{l\to \infty}\frac{-1}{(\log |t_l|)^{2k+1}}\int_{N_{\alpha,t_l}}\parallel \tilde{s}_{t_l}\parallel_{\text{KE}}^2 \omega_t^n= \parallel s \parallel_{\{t_l \},D_\alpha}^2.$$
  
\end{proposition}

We normalize the left hand side of the formula by 
$$\parallel \cdot \parallel^2_{t,\text{KEN}}=\frac{-1}{(\log |t|)^{2k+1}}\int_{X_t}\parallel \cdot \parallel_{\text{KE}}^2 \omega_t^n,$$ for sections in $H_0(X_t,(k+1)K_{X_t})$. Then we have the following.
\begin{corollary}\label{cor-inn-ai}
    For $s\in H^0(D_\alpha,(k+1)L)$ such that $\parallel s\parallel_{\{t_l \},D_\alpha}=1$, we have
    $$\lim_{l\to \infty}\parallel \tilde{s}_{t_l,\text{mod}} \parallel^2_{t_l,\text{KEN}}=1.$$
\end{corollary}
The corollary implies that the map $s\mapsto \tilde{s}_{t_l,\text{mod}} $ from $(H^0(D_\alpha,(k+1)L),\parallel \cdot\parallel_{\{t_l \},D_\alpha})$ to $(H_0(X_{t_l},(k+1)K_{X_{t_l}}),\parallel \cdot \parallel^2_{t_l,\text{KEN}})$ are almost-isometric embeddings for $l$ large.

\subsection{Outer sections}

%If ${D_\alpha}=X_{0,i}\cap X_{0,j}$, Let $\bar{V}'_{a}$($\bar{V}''_{a}$ respectively) be the space of minimal extensions of $H^0({D_\alpha},kL-[X_{0,i}])$ to $H^0(\xcal,kL-[X_{0,i}])$($H^0({D_\alpha},kL-[X_{0,j}])$ to $H^0(\xcal,kL-[X_{0,j}])$ respectively). Then let $\tilde{V}_a'=\bar{V}_a'\otimes S_i\subset H^0(\xcal,kL)$ and $\tilde{V}_a''=\bar{V}_a''\otimes S_j\subset H^0(\xcal,kL)$. Then we let $V'_{t,a}\subset \hcal_{t,k}$($V''_{t,a}$ respectively) be the restriction of sections in $\tilde{V}_a'$($\tilde{V}_a''$ respectively) to $X_t$.

%For fixed $\delta>0$, we denote by $K_\delta$ the closed interval $[\delta,1-\delta]$. The image is the annulus $|t|^{1-\delta}\leq |z_1|\leq |t|^{\delta}$.

We denote by $S_{\hat{i} }=\otimes_{j\neq i}S_j\in H^0(\xcal, \sum_{j\neq i}[X_{0,j}])$. 
For each $s^{i1}\in \gcal_{0,i,k+1}$, $$a_i=\frac{s^{i1}}{S_{\hat{i} }}\in H^0(X_{0,i},(k+1)L-\sum_{j\neq i}[X_{0,j}]).$$
Then let $\tilde{a_i}$ be the minimal extension to $\xcal$ of $a_i$. Then we have a section $\tilde{a_i}\otimes S_{\hat{i} }$ in $H^0(\xcal, (k+1)L)$ whose restriction to $X_0$ is $s^{i1}$. Then we denote by $\tilde{s^{i1}}_t\in \hcal_{t,k+1}$ the restriction of $\tilde{a_i}\otimes S_{\hat{i} }$ to $X_t$.

Similarly, for each $s^{i2}\in \hcal_{0,i,k+1,2}$, $$b_i=\frac{s^{i2}}{S^{\otimes 2}_{\hat{i} }}\in H^0(X_{0,i},(k+1)L-2\sum_{j\neq i}[X_{0,j}]).$$
Then let $\tilde{b_i}$ be the minimal extension to $\xcal$ of $b_i$. Then we have a section $\tilde{b_i}\otimes S^{\otimes 2}_{\hat{i} }$ in $H^0(\xcal, (k+1)L)$ whose restriction to $X_0$ is $s^{i2}$. Then we denote by $\tilde{s^{i2}}_t\in \hcal_{t,k+1}$ the restriction of $\tilde{b_i}\otimes S^{\otimes 2}_{\hat{i} }$ to $X_t$.

Fix an orthonormal basis $\bcal_{0,i,1}=\{s^{i1}_l\}$ for $\gcal_{0,i,k+1}$ and $\bcal_{0,i,2}=\{s^{i2}_l\}$ for $\hcal_{0,i,k+1,2}$ for all $i$. We get sections $\bcal_{1t,1}=\{\tilde{s}^{i1}_{t,l}\}$ and $\bcal_{1t,2}=\{\tilde{s}^{i2}_{t,l}\}$ in $\hcal_{t,k+1}$. We then give these sections $\bcal_{1t}=\bcal_{1t,1}\cup \bcal_{1t,2}$ an arbitrary order. And we denote by $\qcal_{1t}$ the Hermitian matrix defined by the inner products of these sections defined by $\omega_t$. We now show that
$$\lim_{t\to 0}\qcal_{1t}=I,$$
namely these sections are almost orthonormal.

\

For a fixed $\delta>0$, we denote by $X_t^{\delta}$ the points on $X_t$ whose distance to $D$ is bigger than $\delta$. We can similarly define $X_{0,i}^{\delta}$. 
Then when $t$ is small enough, for each point $q$ on $X_t^{\delta}$, there is exactly one $i$ such that $d(q,X_{0,i})=d(q,X_{0})$. So $X_t^{\delta}$ has $m$ connected components $\{X_t^{\delta,i}\}_{1\leq i\leq m}$. We also arrange the ordering of the components so that $X_t^{\delta,i}$ is closest to $X_{0,i}$. The complement of $X_t^{\delta}$ in $X_t$ also has connect components $\{C_t^{\delta,\alpha}\}_{\alpha\in \Lambda}$ satifying $D_\alpha\subset C_t^{\delta,\alpha}$ for each $\alpha\in \Lambda$.

We fix $\delta$ so that $\log |\delta|^2=-M_1\leq-4M$. 

Then for each $j\neq i$, on $X_t^{\delta,i}$, we have 
\begin{equation}\label{e-ct}
    \parallel \tilde{s}^{ij}_{t,l}\parallel^2_{\text{KE}}=O(|t|^{2j}), \quad j=1,2.
\end{equation}
 
When $D_\alpha \subseteq X_{0,i}$, let $D_\alpha= X_{0,i}\cap X_{0,q}$. Again, for simplicity, we can that $i=1, q=2$As before, let $U$ be one neighborhood centered at $p\in D_\alpha$ with coordinates $(z_0,\cdots,z_n)$ and such that $X_{0,1}\cap U=\{z\big| z_0=0 \}$. Then $\tilde{s}^{i1}_{l}=f(dz_0\wedge\cdots\wedge dz_n)^{\otimes(k+1)}$. Then on $X_t$, as before, we can write $f=f_++f_-$ with $f_+=\sum_{q\geq 1}a_q z_1^q$ and $f_-=\sum_{q\leq -1}a_q z_1^q$, where $a_q=a_q(z_2,\cdots,z_n)$. And we define $\tilde{s}^{i1}_{l\pm}=f_{\pm}(dz_0\wedge\cdots\wedge dz_n)^{\otimes(k+1)}$
 Then similar to formula \ref{e-spm}, we have 
 \begin{equation}
    \int_{U\cap C_t^{\delta,\alpha}}\parallel \tilde{s}^{i1}_{t,l-}\parallel^2_{\text{KE}}\omega_t^n=O(|t|^2),
 \end{equation}

and \begin{eqnarray}
    \int_{U\cap C_t^{\delta,\alpha}}\parallel \tilde{s}^{i1}_{t,l+}\parallel^2_{\text{KE}}\omega_t^n&\leq& 2\int_{X_t\cap \{\frac{1}{2}M_1<\sigma<M_1\}}\parallel \tilde{s}^{i1}_{t,l+}\parallel^2_{\text{KE}}\omega_t^n\\ \label{ine-m1}
    &\leq& 3\int_{X_{0,1}\cap \{\frac{1}{2}M_1<\sigma<M_1\}}\parallel s^i\parallel^2_{\text{KE}}\omega_0^n,
\end{eqnarray}
where the last inequality holds for $|t|$ small enough.
And by lemma \ref{lem-concave}, $$\int_{X_t\cap \{\sigma>y \}}\parallel \tilde{s}^{i1}_{t,l+}\parallel^2_{\text{KE}}\omega_t^n=O(e^{-y})=O(|t|).$$
So 
\begin{equation}\label{e-d-s}
    \int_{X_t\cap \{\sigma>y \}}\parallel \tilde{s}^{i1}_{t,l}\parallel^2_{\text{KE}}\omega_t^n=O(e^{-y})=O(|t|).
\end{equation}
When $D_\alpha \nsubseteq X_{0,i}$, we can basically repeat this argument to get that 
\begin{equation}\label{e-d-ns}
    \int_{U\cap C_t^{\delta,\alpha}}\parallel \tilde{s}^{i1}_{t,l}\parallel^2_{\text{KE}}\omega_t^n=O(|t|^2).
\end{equation}
Then we can prove the following proposition.
\begin{proposition}
When $i\neq j$, we have
    $$\langle \tilde{s}^{i1}_{t,l_1},\tilde{s}^{j1}_{t,l_2} \rangle_{\text{KE},X_t}=O(\sqrt{|t|}), $$
    for any $l_1,l_2$.
\end{proposition}
\begin{proof}
    When $X_{0,i}\cap X_{0,j}=\emptyset$, this follows directly from equations \ref{e-ct} and \ref{e-d-ns}.

    When $X_{0,i}\cap X_{0,j}\neq \emptyset$, then by equation \ref{e-ct}, we only need to consider the components $C_t^{\delta,\alpha}$ where $D_\alpha\subset X_{0,i}\cap X_{0,j}$. Then in applying equation \ref{e-d-s} to $\tilde{s}^{j1}_{t,l_2}$, the region $$\{\sigma<y \}$$ becomes $\{\sigma<y \}$, which is in complement to each other. Then the proposition follows. 

\end{proof}
Similar arguments apply to sections of the form $\tilde{s}^{i2}_{t,l}$. Then one can similarly prove the following proposition.
\begin{proposition}
    When the vector $(i_1,j_1)\neq (i_2,j_2)$, we have
        $$\langle \tilde{s}^{i_1j_1}_{t,l_1},\tilde{s}^{i_2j_2}_{t,l_2} \rangle_{KE,X_t}=O(\sqrt{|t|}), $$
        for any $l_1,l_2$.
    \end{proposition}

When $i=j$, we have the following proposition.
\begin{proposition}
    For any $l$, we have
    $$\lim_{t\to 0}\int_{X_t}\parallel \tilde{s}^{i1}_{t,l}\parallel^2_{\text{KE}}\omega_t^n=1.$$
\end{proposition}
\begin{proof}
    For any $\epsilon>0$, $\exists \delta$ so that $$\int_{X_{0,i}\cap \{q\big| d^2(q,D)<\delta \}}\parallel s^i\parallel^2_{\text{KE}}\omega_0^n<\frac{1}{2}\epsilon.$$
   Then by inequality \ref{ine-m1}, we have 
   $$\int_{X_t\backslash X_t^{\delta,i}}\parallel \tilde{s}^{i1}_{t,l}\parallel^2_{\text{KE}}\omega_t^n<\epsilon,$$
   for $|t|$ small enough. Then by theorem \ref{thm-tian}, we have 
$$\big|\int_{X_t^{\delta,i}}\parallel \tilde{s}^{i1}_{t,l}\parallel^2_{\text{KE}}\omega_t^n-\int_{X_{0,i}\backslash \{q\big| d^2(q,D)<\delta \}}\parallel s^i\parallel^2_{\text{KE}}\omega_0^n\big|<\epsilon,$$ for $|t|$ small enough.
   It is not hard to see that the estimates can be made uniform for sections $s^i$ of unit norm, so the proposition follows.
\end{proof}
One can then apply similar arguments to sections of the form $\tilde{s}^{i2}_{t,l}$ to get similar proposition.
\begin{proposition}
    For any $l$, we have
    $$\lim_{t\to 0}\int_{X_t}\parallel \tilde{s}^{i2}_{t,l}\parallel^2_{\text{KE}}\omega_t^n=1.$$
\end{proposition}
Then the mutual orthogonality between the sections follows. So we have proved the following.
\begin{theorem}\label{thm-q1t}
    We have $$\lim_{t\to 0}\qcal_{1t}=I.$$
\end{theorem}

Let $\{t_u\}_{u=1}^{\infty}$ be a good sequence. Following corollary \ref{cor-inn-ai}, for each $\alpha$, we pick an orthonormal basis $\bcal^{\alpha}=\{s^{\alpha}_l\}$ of $(H^0(D_\alpha,(k+1)L),\parallel \cdot\parallel_{\{t_u \},D_\alpha})$. And we denote by $$^{t_u}s^{\alpha}_l\triangleq (-\log |t_u|)^{-\frac{2k+1}{2}}(\widetilde{s^{\alpha}_l})_{t_u,\text{mod}}. $$ 

And we have a set of sections $$\bcal_{2t_u}=\{^{t_u}s^{\alpha}_l\}_{\alpha,l}.$$
We fix an order on $\Lambda$, and then order $\bcal_{2t_u}$ according to the tuple $(\alpha,l)$ in dictionary order.

Then by corollary \ref{cor-inn-ai}, we have $$\lim_{u\to\infty}\langle ^{t_u}s^{\alpha}_{l_1},^{t_u}s^{\alpha}_{l_2} \rangle_{\text{KE}}=\delta^{l_1}_{l_2},$$
where $\delta^{l_1}_{l_2}$ is the Kronecker symbol.
And by proposition \ref{prop-sqrty}, we have that for $\alpha\neq \beta$, $$\lim_{u\to\infty}\langle ^{t_u}s^{\alpha}_{l_1},^{t_u}s^{\beta}_{l_2} \rangle_{\text{KE}}=0, $$ for any $l_1,l_2$. So the corresponding Hermitian matrix $\qcal_{2t_u}$ satisfies $$\lim_{u\to \infty }\qcal_{2t_u}=I.$$
%$H^0(D_\alpha,(k+1)L)$ $s\mapsto \tilde{s}_{t_l,\text{mod}} $ from  to $(H_0(X_{t_l},(k+1)K_{X_{t_l}}),\parallel \cdot \parallel^2_{t_u,\text{KEN}})$

Similar to equation \ref{e-d-s}, we have
\begin{equation*}
    \int_{X_t\cap \{\sigma>\sqrt{y} \}}\parallel \tilde{s}^{i1}_{t,l}\parallel^2_{\text{KE}}\omega_t^n=O(e^{-\sqrt{y}}).
\end{equation*}
Then by proposition \ref{prop-sqrty}, we have that $$\lim_{u\to\infty}\langle ^{t_u}s^{\alpha}_{l},s \rangle_{\text{KE},X_t}=0, $$ for any $s\in \bcal_{1t_u}$ and for any $l$. Then since $\#\bcal_{1t_u}+\#\bcal_{2t_u}=\dim \hcal_{t_u,k+1}$, one sees that $\bcal_{1t_u}\cup \bcal_{2t_u}$ form a basis for $\hcal_{t_u,k+1}$ which is almost orthonormal.

\section{Bergman embeddings}\label{sec-4}
 
We order the basis $\bcal_{0,i}=\bcal_{0,i,2}\cup \bcal_{0,i,1}$ of $\hcal_{0,i,k+1}$ from $\bcal_{0,i,2}$ to $\bcal_{0,i,1}$, namely elements of $\bcal_{0,i,2}$ goes before elements in $\bcal_{0,i,1}$. Then this defines a Kodaira embedding $\Phi_{0,i,k+1}:X_{0,i}\to \CP^{n_{k+1,i}-1}$. Also we use $\bcal^{\alpha}$ to define a Kodaira embedding $\Psi^{\alpha}_{k+1}:D_\alpha\to \CP^{d_{\alpha,k+1}-1}$.
Then we order $\bcal_{0}=\cup_i\bcal_{0,i}$ according to the order of $\{i\}$.
 
We can then order $\bcal_{1t_u}$ according to the order of $\bcal_{0}$. 
Define $\Psi_{t_u,k+1}:X_{t_u}\to \CP^{N_{k+1}-1}$ by the basis $\bcal_{1t_u}\cup \bcal_{2t_u}$ ordered from $\bcal_{1t_u}$ to $\bcal_{2t_u}$. We also define the linear embedding $$A_{0,i}:\CP^{n_{k+1,i}-1}\to \CP^{N_{k+1}-1},$$ considering the index set of $\bcal_{0,i}$ as a subset of the index set of $\bcal_{1t_u}\cup \bcal_{2t_u}$. And similarly we define 
%$$A_{0}:\CP^{n_{k+1}-1}\to \CP^{N_{k+1}-1},$$and
  $$A_{\alpha}:\CP^{d_{\alpha,k+1}-1}\to \CP^{N_{k+1}-1}$$
For simplicity, we will use $\Phi_{0,i,k+1}$ in place of $A_{0,i}\circ\Phi_{0,i,k+1}$, and $\Psi^{\alpha}_{k+1}$ in place of $A_{\alpha}\circ\Psi^{\alpha}_{k+1}$.

Denote by $\rho_{t,k+1}$ the Bergman kernel function of $(X_t,(k+1)L,\omega_t)$.
For any fixed $\delta>0$, by theorem \ref{thm-q1t}, there exists positive lower bound for $\rho_{t,k+1}$ on $X_t^\delta$. Then by proposition \ref{prop-sqrty}, the norms of elements in $\bcal_{2t_u}$ goes to $0$ on $X_{t_u}^{\delta}$ as $u\to\infty$. Therefore, we have 
$$\lim_{u\to \infty} \Psi_{t,k+1}(X_{t_u}^{\delta,i})=\Phi_{0,i,k+1}(X_{0,i}^{\delta}),$$for each $i$.

\

We consider the points most close to $D$. Let $U$ be an appropriate coordinates patch centered at $p\in D_\alpha=X_{0,1}\cap X_{0,2}$. Then $d^2(q,X_{0,1})\approx a_0|z_0|^2$ and $d^2(q,X_{0,2})\approx a_1|z_1|^2$ for some $a_0>0, a_1>0$ when $|z_0|^2+|z_1|^2$ is small. So $\exists C>0$ such that when $|\frac{z_0}{z_1}|>C$, the point $q=(z_0,\cdots,z_n)$
satisfies $d(q,X_{0,1})=d(q,X_{0})$. Since $|z_0z_1|=|t|$, $|\frac{z_0}{z_1}|=|\frac{z_0^2}{t}|$. $\exists C_1$ such that $x+\frac{1}{x}>C_1$ implies $x>C$ or $\frac{1}{x}>C$. So when
$|z_0|^2+|z_1|^2>C_1|t|$, we have $|\frac{z_0}{z_1}|>C$ or $|\frac{z_1}{z_0}|>C$. So $\exists C_2>0$ such that $X_t^{C_2\sqrt{|t|}}=\cup_i X_t^{C_2\sqrt{|t|},i}$ has $m$ components. And we can order these components as for $\cup_i X_t^{\delta,i}$ for fixed small $\delta$. Then the complement $X_t\backslash X_t^{C_2\sqrt{|t|}}$ also decomposes to connected components $\cup_{\alpha\in \Lambda}C_t^{C_2\sqrt{|t|},\alpha}$.

Let $\rho_{\{t_u\},\alpha,k+1}$ denote the Bergman kernel of $(D_\alpha,(k+1)L,\vcal, d\nu_{\{t_u\},\alpha,k+1})$, then clearly $\exists C_3$, possibly depending on $(\vcal, d\nu_{\{t_u\},\alpha,k+1})$, such that $\rho_{\{t_u\},\alpha,k+1}>C_3$. If we denote by $\zeta_{\{t_u\},\alpha,k+1}(q)$ the sum $$\sum_l \parallel ^{t_u}s^{\alpha}_l \parallel^2_{\text{KE}}.$$
Then on $C_t^{C_2\sqrt{|t_u|},\alpha}$, we have \begin{equation}\label{e-zeta}
    \zeta_{\{t_u\},\alpha,k+1}(q)>C_4\big| \log |t_u| \big|, 
\end{equation}
for some $C_4>0$. On the other hand, if $S=\tilde{s}^{i1}_{t_u,l}$ for some $i,l$, then on $C_t^{C_2\sqrt{|t_u|},\alpha}$, we have 
$$\parallel S \parallel^2_{\text{KE}}=O(|t_u|(\log |t_u|)^{2k+2}). $$
This can be seen as follows. When $D_\alpha\subset X_{0,i}$, namely $D_\alpha X_{0,i}\cap X_{0,j}$ for some $j$, then by the construction of $\tilde{s}^{i1}_{t_u,l}$, $$\parallel\frac{S}{S_{\hat{i}}}\parallel_h<C_5,$$
for some $C_5>0$ independent of $t$. So use the coordinates as before, 
$S$ can be represented by $z_1f(z)$. Then $|z_1|^2=O(|t|)$ on $C_t^{C_2\sqrt{|t|},\alpha}$, and the conclusion follows. The case when $D_\alpha\nsubseteq X_{0,i}$ is similar.

Similarly, if $S=\tilde{s}^{i2}_{t_u,l}$ for some $i,l$, then on $C_t^{C_2\sqrt{|t|},\alpha}$, we also have 
$$\parallel S \parallel^2_{\text{KE}}=O(|t|^2(\log |t|)^{2k+2}). $$

Therefore we have the following proposition.
\begin{proposition}
    We have $$\lim_{u\to\infty }\Psi_{t_u,k+1}(C_{t_u}^{C_2\sqrt{|t_u|},\alpha})=\Psi^{\alpha}_{k+1}(D_{\alpha}),$$ for each $\alpha$.
\end{proposition}

And we denote by $\rho_{t,i,1,k+1}$ the sum $$\sum_l \parallel\tilde{s}^{i1}_{t,l} \parallel^2_{\text{KE}}.$$
Then similarly, on the region $$T_{t,i}\triangleq X_t^{\frac12\big|\log |t|\big|^{-k},i}\backslash X_t^{2 \big|\log |t|\big|^{-k},i}, $$
we have $$\rho_{t,i,k+1}>C_6\big|\log |t|\big|^{-2k}a_t^{2k+2},$$
where $a_t=\log\big|\log |t|\big| $, for some $C_6>0$. And for any $s_1\in \bcal_{1t,2}$ and $s_2\in \bcal_{2t}$, we have 
$$\parallel s_1 \parallel_{\text{KE}}^2=O(\big|\log |t|\big|^{-4k}a_t^{2k+2}),$$ 
and $$\parallel s_2 \parallel_{\text{KE}}^2=O(\big|\log |t|\big|^{-2k-1}a_t^{2k+2}),$$ 
on $T_{t,i}$. Consequently, one sees that the sections in $B_{2t_u}$ has negligible influence on the part $X_{t_u}^{2\big|\log |{t_u}|\big|^{-k},i}$ and that the sections in $\bcal_{1{t_u},2}$ has negligible influence on the part $X_{t_u}^{C_2\sqrt{|{t_u}|},i}\backslash X_{t_u}^{2\big|\log |{t_u}|\big|^{-k},i}$. 

If we denote by $$T_{0,i,t}\triangleq X_{0,i}^{\frac12 \big|\log |t|\big|^{-k}}\backslash X_{0,i}^{2 \big|\log |t|\big|^{-k}} ,$$
then clearly $$\lim_{t\to 0}\Phi_{0,i,k+1}(T_{0,i,t})=\Phi_{0,i,k+1}(D^i). $$
Therefore, we have 
$$\lim_{u\to\infty}\Psi_{t_u,k+1}(T_{t_u,i})=\Phi_{0,i,k+1}(D^i). $$
And we also have $$\lim_{u\to \infty}\Psi_{t_u,k+1}(X_{t_u}^{2\big|\log |{t_u}|\big|^{-k},i})=\Phi_{0,i,k+1}(X_{0,i}). $$

This left us with the parts $X_{t_u}^{C_2\sqrt{|{t_u}|},i}\backslash X_{t_u}^{2\big|\log |{t_u}|\big|^{-k},i}$. 

\

Back to an appropriate coordinates patch $U$ centered at a point $p\in D_\alpha\subset X_{0,i}$ such that $X_{0,i}=\{z_0=0\}$. By a unitary transformation, we can assume that in the standard frame $(dz_0\wedge\cdots\wedge dz_n)^{\otimes (k+1)}$, $s^{\alpha}_1(p)\neq 0$ while $s^{\alpha}_l(p)=0$ for $l>1$. Let $E_p$ be the $(z_0,z_1)$-plane passing through $p$, namely the points $q$ such that $z_j(q)=z_j(p)$ for $j\geq 2$.
Recall that in the construction of $^{t_u}s^{\alpha}_l$, we need to substract a solution $v_l$ (the index $l$ is added for distinction) to a certain $\dbar$-equation. 
Since $v_l$ is holomorphic around $D_\alpha$, it is represented by a holomorphic function $f(z)=a_0+\sum_{i>0}a_iz_1^i+\sum_{i<0}a_iz_1^i$. Then we can use arguments similar to those for equation \ref{e-tils-pm} to get from the $L_2$ estimate $$\int_{X_t}\parallel v_l\parallel^2_{\text{KE}}\omega_t^n<c_M\big|\log |t|\big|^{-k}$$
that on $\lcal_{p,t_u}$, where $\lcal_{p,t_u}$ is the intersection of $(X_{t_u}^{C_2\sqrt{|{t_u}|},i}\backslash X_{t_u}^{2\big|\log |{t_u}|\big|^{-k},i})$ with $E_p$, the point wise norm $\parallel v_l\parallel^2_{\text{KE}}$ is negligible compared to $\parallel {^{t_u}s^{\alpha}_1}\parallel^2_{\text{KE}}+\rho_{t_u,i,1,k+1}$, namely the quotient goes to $0$ as $u\to \infty$. A different way to see this is by noticing that $\bcal_{t_u}=\bcal_{1t_u}\cup \bcal_{2t_u}$ is an almost orthonormal basis for $\hcal_{t_u,k+1}$, which implies that small $L_2$-norm means small pointwise norm compared to the almost-Bergman kernel:$$\sum_{s\in \bcal_{t_u}}\parallel s\parallel^2_{\text{KE}} .$$
So the holomorphic function of $^{t_u}s^{\alpha}_1$ can be considered as a non-zero constant on $\lcal_{p,t_u}$. 
Similarly, on $\lcal_{p,t_u}$, the point wise norm $\parallel \tilde{s}^{j1}_{t_u,l}\parallel^2_{\text{KE}}$ for $j\neq i$ is negligible compared to $\parallel {^{t_u}s^{\alpha}_1}\parallel^2_{\text{KE}}+\rho_{t_u,i,1,k+1}$.

Then we can apply an unitary transformation to make $\frac{\tilde{s}^{i1}_{t_u,1}}{S_{\hat{i}}}(p)\neq 0$ while $\frac{\tilde{s}^{i1}_{t_u,l}}{S_{\hat{i}}}(p)= 0$ for $l> 1$. Then similarly we have 
$\parallel \tilde{s}^{i1}_{t_u,l}\parallel^2_{\text{KE}}$ for $l>1$ is negligible compared to $\parallel {^{t_u}s^{\alpha}_1}\parallel^2_{\text{KE}}+\rho_{t_u,i,1,k+1}$ on $\lcal_{p,t_u}$.
So now on $\lcal_{p,t_u}$, we only need to look at two sections $^{t_u}s^{\alpha}_1$ and $\tilde{s}^{i1}_{t_u,1}$, where the local holomorphic function of the latter can be written as $a_1z_1f(z)$, where $f(z)$ is a holomorphic function that converges to the local holomorphic representation of $\frac{{s}^{i1}_{t_u,1}}{S_{\hat{i}}}$ on $X_{0,i}\cap U$.

Then it is clear that the images $\Psi_{t_u,k+1}(\lcal_{p,t_u})$, as $u\to\infty$, converges to the linear $\CP^1$-line connecting $\Psi_{k+1}^{\alpha}(p)$ to $\Phi_{0,i,k+1}(p)$. Therefore, the image of $$X_{t_u}^{C_2\sqrt{|{t_u}|},i}\backslash X_{t_u}^{2\big|\log |{t_u}|\big|^{-k},i}$$ satisfies the description of the components $Y_{\alpha,1}$ and $Y_{\alpha,2}$ in the statement of the main theorem. Therefore we have proved the theorem except that the basis $\bcal_{t_u}=\bcal_{1t_u}\cup \bcal_{2t_u}$ of $\hcal_{t_u,k+1}$ is not orthonormal. We then only need to apply a Gram-Schmidt process to produce a genuine orthonormal basis $\bcal_{t_u}'$. Then since the inner product matrix of $\bcal_{t_u}$ converges to the identity, one sees that sequence of images of the Bergman embeddings $\Phi_{t_u,k+1}$ induced by $\bcal_{t_u}'$ converges to the same as the limit for $\Psi_{t_u,k+1}$. And we have proved the main theorem.

\

\subsection*{Embedding of normal bundle}
Let $\pi:A\to M$ be a line bundle over a connected compact complex manifold. Let $\hat{A}=A\cup M_\infty$ be the projective completion of the total space $A$. Let $L$ be a positive line bundle over $Y$.
Denote by $\pi_\infty:\hat{A}\to M_\infty$ the projection and by $L$ on $\hat{A}$ the pull-back $\pi^{-1}(L)$. And we use $M\subset \hat{A}$ to denote the zero section of $A$. We fix a section $S_M\in H^0(\hat{A},[M])$. Then we have exact sequence $$0\to H^0(\hat{A},kL+[M_\infty]-[M])\to H^0(\hat{A},kL+[M_\infty])\to H^0(M,kL+[M_\infty]),$$
where the first non-zero morphism is given by $s\mapsto s\otimes S_M$.
Since $kL+[M_\infty]-[M]$ is isomorphic to $\pi_\infty^{*}((kL+[M_\infty]-[M])\big|_{M_{\infty}})$, we have $$H^0(\hat{A},kL+[M_\infty]-[M])\simeq H^0(M_{\infty},kL+[M_\infty]-[M])\simeq H^0(M_{\infty},kL+[M_\infty]).$$ And since $H^i(\CP^1,\ocal)=0$, for $i>1$, we have vanishing higher direct images $R^i\pi_*(kL+[M_\infty]-[M])=0$ for $i>0$. And since $\pi_*(kL+[M_\infty]-[M])\simeq kL+[M_\infty]-[M]$, we have $H^1(\hat{A},kL+[M_\infty]-[M])\simeq H^1(M_{\infty},kL+[M_\infty]-[M])$. Therefore, for $k$ large enough, we have exact sequence $$0\to H^0(M_\infty,kL+[M_\infty])\to H^0(\hat{A},kL+[M_\infty])\to H^0(M,kL+[M_\infty])\to 0.$$ 
We have identifications $I_\infty=\pi_\infty\big|_{M}:M\to M_\infty $,  $$H^0(M_\infty,kL+[M_\infty])\simeq H^0(M_\infty,kL-A)$$ and $H^0(M,kL+[M_\infty])\simeq H^0(M,kL).$
So for $k$ large enough, a basis of $H^0(\hat{A},kL+[M_\infty])$ defines a Kodaira embedding $\phi_k:\hat{A}\to \PP H^0(\hat{A},kL+[M_\infty])$.  And  one sees that the image $\phi_k(\hat{A})$ can be described as the union of linear $\CP^1$'s each of which connects $\phi_{k,0}(p)$ to $\phi_{k,\infty}(p)$ for some $p\to Y$, where $\phi_{k,0}$ is the restriction of $\phi$ to $Y$ and $\phi_{k,\infty}$ is the composition $\phi_k\circ I_\infty$. 

Let $\pi_1:N_{\alpha,1}\to D_\alpha$ be the normal bundle of $D_\alpha$ in $X_{0,i}$. Let $\hat{N}_{\alpha,1}=N_{\alpha,1}\cup D_\alpha^\infty$ be the projective completion. Then it is clear that $Y_{\alpha,1}$ in the statement of the main theorem is the image of $\hat{N}_{\alpha,1}$ defined by a basis of $H^0(\hat{N}_{\alpha,1},k\pi_1^{-1}L+[D_\alpha^\infty])$.

\

\begin{proof}[Proof of theorem \ref{thm-2}.]
    Let $\{t_u\}$ be a good sequence. Let $\rho_{{\alpha,\{t_u\}},k+1}$ denote the Bergman kernel of $$(H^0(D_\alpha,(k+1)L),\vcal,d\nu_{\{t_l\},\alpha,k}). $$ 
    Then in an appropriate coordinates patch $U$, by theorem \ref{the-zhang-3.3}, proposition \ref{prop-inner} and by our previous arguments, the Bergman kernel $\rho_{t_u,k+1}$ at a point $\{|z_0|=|z_1|\}\cup X_{t_u}$ is close to $$\rho_{{\alpha,\{t_u\}},k+1}\big|\log |t| \big|^{-(2k+1)+(2k+2)}.$$ 
    Since $D_{\alpha}$ is compact, $\log \rho_{{\alpha,\{t_u\}},k+1}$ is bounded. Then one can argue by contradiction. For the lower bound, assume that there is a sequence $a_l\to 0$ such that  $$\lim_{l\to\infty }\frac{\lambda_u(a_l,k)}{\big| \log |a_l| \big|}=0.$$ 
    Then a subsequence $\{t_u\}$ is a good sequence. Then we get that $\log \rho_{{\alpha,\{t_u\}},k+1}$ is not bounded below, a contradiction. The arguments for the upper bound is similar.
    Then the $\lambda_u$ part of the theorem follows.

    For the $\lambda_l$ part, one just needs to estimate $\rho_{t_u,k+1}$ on the points where $|z_1|^2=\big| \log |t_u| \big|^{-2k-1}$ on an appropriate patch $U$. In fact, as in the proof of theorem \ref{thm-main}, this is reduced to the computation on the set $\lcal_{p,t_u}$. Then it is clear that on the points where $|z_1|^2=\big| \log |t_u| \big|^{-2k-1}$, we have $$\rho_{t_u,k+1}=O(\big| \log |t_u| \big|^{-2k-1}a_{t_u}^{2k+2}).$$  
    So we have proved theorem \ref{thm-2}.

    Notice that equation \ref{e-zeta} is not the same as what claimed the comments after theorem \ref{thm-2}, but one can apply an argument by contradiction again to get that $\rho_{t,k}>C_1\big|\log |t| \big|$ if $d(p,D)<C_2\sqrt{|t|}$.
\end{proof}

\

\begin{proof}[Proof of the claim that for $k$ large enough, $\pi_* \ocal(kL)$ is locally free on $B$].
    
    It suffices to show that any $s\in H^0(X_0,kL)$ can be extended to $\xcal$. We have the following exact sequence:
    $$0\to H^0(X_0,kL-D)\to H^0(X_0,kL)\to H^0(D,kL)\to 0, $$
    where the inclusion is given by $s\mapsto s\otimes s_D$ for $s\in H^0(X_0,kL-D)$. Given $s\in H^0(X_0,kL)$, we denote by $s_2$ the image of $s$ in $H^0(D,kL)$. Then for $k$ large enough, we know that $s_2$ can be extended to a section $\tilde{s}_2$ in $H^0(\xcal, kL)$. Then let $s_1=s-\tilde{s}_2\in H^0(X_0,kL)$. For each $i$, $s_1|_{X_{0,i}}$ can be extended to $\tilde{s}_{1,i}\in H^0(\xcal, kL)$. Then $\tilde{s}_1\triangleq \sum \tilde{s}_{1,i}$ is an extension of $s_1$. Then $\tilde{s}_1+\tilde{s}_2\in H^0(\xcal, kL)$ is an extension of $s$. 
\end{proof}
\bibliographystyle{plain}
\bibliography{references}
\end{document}